\documentclass[11pt]{amsart}
\usepackage{amssymb}
\usepackage{amsfonts}
\usepackage{color}

\usepackage{xy}
\xyoption{all}

\newtheorem{theorem}{Theorem}[section]
\newtheorem{proposition}[theorem]{Proposition}
\newtheorem{lemma}[theorem]{Lemma}
\newtheorem{corollary}[theorem]{Corollary}

\newtheorem{example}[theorem]{Example}

\theoremstyle{remark}
\newtheorem{remark}[theorem]{Remark}

\renewenvironment{proof}{{\noindent\bf Proof.}}{\hfill $\Box$\par\vskip3mm}

\newcommand{\im}{{\rm Im}\,}

\newcommand{\Hom}{{\rm Hom}}

\newcommand{\Ext}{{\rm Ext}}

\newcommand{\Aa}{\mathcal{A}}

\newcommand{\Cc}{\mathcal{C}}

\newcommand{\Ff}{\mathcal{F}}

\newcommand{\Mm}{\mathcal{M}}

\def\NN{{\mathbb N}}

\def\KK{{\mathbb K}}

\begin{document}
\title{On Extensions of Rational Modules}

\begin{abstract}
We investigate when the categories of all rational $A$-modules and of finite dimensional rational modules are closed under extensions inside the category of $C^*$-modules, where $C^*$ is the cofinite topological completion of $A$. We give a complete characterization of these two properties, in terms of a topological and a homological condition. We also give connections to other important notions in coalgebra theory such as coreflexive coalgebras. In particular, we are able to generalize many previously known partial results and answer some questions in this direction, and obtain large classes of coalgebras for which rational modules are closed under extensions as well as various examples where this is not true.
\end{abstract}

\author{Miodrag Cristian Iovanov\\}
\address{ 
University of Southern California, 3620 South Vermont Ave. KAP 108\\
Los Angeles, CA 90089, USA \&\\
University of Bucharest, Fac. Matematica \& Informatica,
Str. Academiei 14,
Bucharest 010014,
Romania, 
}
\email{yovanov@gmail.com, iovanov@usc.edu}

\date{}
\maketitle

\section*{Introduction and Preliminaries}

Let $\Cc$ be an abelian category and $\Aa$ be a full subcategory of $\Cc$. We say that $\Aa$ is closed if it is closed under subobjects, quotients and direct sums (coproducts). With any closed subcategory $\Aa$ of $\Cc$, there is an associated trace functor (or preradical) $T:\Cc\rightarrow \Aa$ which is right adjoint to the inclusion functor $i:\Aa\rightarrow \Cc$; that is $T(M)$=the sum of all subobjects of $M$ which belong to $\Aa$. Classical examples include the torsion group of an abelian group or the or more generally the torsion part of an $R$-module for a commutative domain $R$, or the singular torsion, which is, for a left $R$-module $M$, defined as $Z(M)=\{x\in M| {\rm ann}_R(x){\rm\,is\,an\,essential\,ideal}\}$ (see \cite{G72}). 

Another important example is that of rational modules: given an algebra $A$, we call a module rational if it is a sum (colimit) of its finite dimensional submodules. The category of rational $A$-modules is equivalent to the category of right $C$-comodules $\Mm^C$, where $C=R(A)=A^0$ is the coalgebra of representative functions on $A$ or the finite dual algebra of $A$ (\cite[Chapter 1]{DNR}). The dual of the coalgebra is again an algebra, and the category of rational $A$-modules is a closed subcategory of left $C^*$-modules ${}_{C^*}\Mm=C^*-{\rm Mod}$. In fact, there is a morphism $A\rightarrow C^*$, and $C^*$ can be thought as a completion of $A$ with respect to the linear topology having a basis of neighborhoods of $0$ consisting of finite ideals of $A$ (see \cite{Tf}).

The above described situation has roots in algebraic geometry. Let $G$ be an affine algebraic group scheme over an algebraically closed field $K$ of positive characteristic $p$ and let $A$ be the Hopf algebra representing $G$ as a functor from {\it Commutative Algebras} to {\it Groups} ($A$ is the "algebra of functions" of $G$). If $M$ be the augmentation ideal of $A$, one defines $M_n:=A\{x^{p^n}|x\in M\}\subseteq A$. Then there is a sequence of finite dimensional Hopf algebras $A/M_n$, with canonical projections $A\rightarrow A/M_n\rightarrow A/M_s$ for $n\geq s$. Geometrically, these correspond to the n'th power of the Frobenius morphism of the scheme. The dual family of finite dimensional Hopf algebras $(A/M_n)^*$ together with the morphisms $(A/M_s)^*\hookrightarrow (A/M_n)^*$ forms an inductive family of Hopf algebras, and the algebra 
$$B=\lim\limits_{\stackrel{\longrightarrow}{n}}(A/M_n)^*$$
is called the {\it hyperalgebra} of $G$. The category of finite dimensional $G$-representations is equivalent to the category of finite dimensional $A$-comodules. Note that $B$ embedds in $A^*$ canonically as algebras, and so every $A$-comodule is a $B$-module, and hence, the category of (rational) $G$-modules or $A$-comodules is a closed subcategory of the category of $B$-modules (modules over the hyperalgebra of $G$). We refer the reader to \cite{FP,Su} and the classical text \cite{J} for further details.


Given a closed subcategory $\Aa$ of $\Cc$, one is often interested also in the situation when $\Aa$ is also closed under extensions, i.e. if $0\rightarrow M'\rightarrow M\rightarrow M''\rightarrow 0$ is a sequence in $\Cc$ with $M',M''$ in $\Aa$, then $M$ is in $\Cc$. In this case, $\Aa$ is usually called a Serre subcategory of $\Cc$. The property of being closed under extensions has also been called localizing, for an obvious reason: in this situation, one can form the quotient category $\Cc/\Aa$ (which is a localization of $\Cc$), see \cite{G}. The property of $\Aa$ being closed under extensions is easily seen to be equivalent to $T$ being a radical, i.e. to the property $T(M/T(M))=0$ for all objects $M\in \Cc$. Examples of localizing categories are torsion modules over a commutative domain, or for the singular torsion modules over a ring $R$ (under an additional assumption that the singular ideal of $R$ is $0$; \cite{G72}), but, in general, a subcategory need not be localizing. For example, in $R-Mod$ one could consider the sub-category of semisimple modules in the situation when $R$ is not semisimple. It is a natural question to ask when the above mentioned rational subcategory of of $C^*$-modules is closed under extensions, or equivalently, when is the functor $Rat: {}_{C^*}\Mm\longrightarrow \Mm^C=Rat({}_{C^*})$ a radical; such a coalgebra is said to have a (left) torsion Rat functor. This question was investigated by many authors \cite{C0, CNO, NT1, L, L2, Rad, HR74, Sh, TT}; it is known, for example, that if $C$ is a right semiperfect coalgebra, then it has a (left) torsion Rat functor \cite{L,NT1}, or if $C$ is such that $C^*$ is left $\Ff$-Noetherian (meaning that left ideals which are closed in the finite topology of $C^*$ are finitely generated) then it also has a torsion Rat functor (\cite{Rad, CNO}). On the other hand, if some aditional conditions on $C$ are satisfied, then one can give an equivalent caracterization of this property: if the coradical $C_0$ of $C$ is finite dimensional (such a coalgebra is called almost connected) then $C$ has a torsion (left or right) Rat functor if and only if every cofinite left (or, equivalently in this case, right) ideal of $C^*$ is finitely generated (see \cite{C0, CNO}, and \cite{HR74} for an equivalent formulation), which is further equivalent to all the terms of the coradical filtration being finite dimensional. In fact, as noted in \cite{CNO}, all the known classes of examples of coalgebras for which rational left $C^*$-modules are closed under extensions turned out to be $\Ff$-Noetherian. This motivated the authors in \cite{CNO} to raise the question of whether the converse is true, i.e. if a coalgebra with torsion (left) rat functor is necesarily (left) $\Ff$-Noetherian. However, the construction in \cite{TT} showed that this is not true. Hence, the problem of completely caracterizing this property remained open, and is perhaps one of the main problems in coalgebra theory, and is important from a more general categorical perspective as pointed out above. 

In this paper, we propose a solution to this problem. There are several main points of our treatment. We provide a complete characterization for when a coalgebra has a left rational torsion functor in Theorem \ref{t.1}, in terms of a topological condition and a homological one. In fact, we first characterize the situation when the finite dimensional rational (left, or right) modules are closed under extensions (inside ${}_{C^*}\Mm=C^*$-Mod); this is equivalent to a topological condition, namely, the set closed cofinite (equivalently, open) ideals of $C^*$ is stable under products. Theorem \ref{t.1} states 

\vspace{.3cm}

{\bf Theorem} \emph{A coalgebra has a left rational torsion functor if and only if the open cofinite ideals are closed under products and $\Ext^1_{C^*}({}_{C^*}C_0,{}_{C^*}C)=0$}

\vspace{.3cm}

We also analyze and give conditions equivalent to the homological property in the above statement. Second, we also provide connections of this property with other important coalgebra notions, such as coreflexive coalgebras. In fact, if the set of simple comodules is a non-measurable set, then the above topological condition is equivalent to the coalgebra being coreflexive, which is a concept developed by several authors \cite{Rad,Rad3,HR74,Tf} (we note that no examples of non-measurable set is known, so this equivalence is in place in any "reasonable" example). Third, this general result allows us to give general classes of examples of coalgebras having a rational torsion functor. In particular, these generalize all the previously known examples, as well as several results characterizing coreflexive, $\Ff$-noetherian or almost noetherian coalgebras (i.e. coalgebras $C$ such that any cofinite ideal in $C^*$ is finitely generated). In particular, we show that:

\vspace{.3cm}

{\bf Theorem} \emph{If the set of cofinite closed ideals is closed under products and one of the following conditions is satisfied, then left rational $C^*$-modules are closed under extensions in ${}_{C^*}\Mm$:\\
-left injective indecomposables have finite coradical filtration;\\
-right injective indecomposables have finite coradical filtration;\\
-for each right injective indecomposable, there is some $n$ such that the $n$'th quotient of the Loewy series (coradical filtration) $L_{n+1}E/L_nE$ is finite dimensional}

\vspace{.3cm}

One other generalization of results on coreflexive coalgebras, $\Ff$-noetherian and almost noetherian coalgebras \cite{C0, HR74} and also on coalgebras with rational torsion functor \cite{CNO} is Theorem \ref{t.3}; this states that 

\vspace{.3cm}

{\bf Theorem} \emph{If in the Ext Gabriel quiver of the left comodules of a coalgebra $C$, vertices have finite left degree, then the following are equivalent:\\
-$E(T)$ has finite dimensional terms in its coradical filtration of all right comodules $T$\\
-in $E(T)^*$ cofinite submodules are finitely generated\\
-$C$ is (left) $\Ff$-noetherian\\
-$C$ is locally finite (see Section 4)\\
-$C$ has a left rational torsion functor} 

\vspace{.3cm}

Fourth, we also give some simple equivalent characterization of the property of $C$ being left $\Ff$-noetherian. With this, we can easily construct simple examples of coalgebras which have a left rational torsion functor but are not $\Ff$-notherian, also answering the aforementioned question of \cite{CNO}. Moreover, the examples show that, while a right semiperfect coalgebra is necesarily left $\Ff$-noetherian, it does not have to be right $\Ff$-noetherian, and also that the $\Ff$-noetherian property is not a left-right symmetric property. 

We close with a few open questions and future possible directions of research. In particular, we ask whether potentially the homological Ext condition might be eliminated from the main result (or, more generally, whether $Ext(C_0,C)$ is always $0$); also, all the known examples of coalgebras with left rational functor, also have a right rational functor, so it is natural to ask whether this is always the case. Based on our general examples, we suggest a possible way to attempt a counter-example, and also to approach the other questions.

\vspace{.5cm}

{\bf Basic properties and notations.}
We refer to the monographs \cite{DNR} and \cite{M} for definitions and properties of coalgebras and rational comodules. We recall here a few basic facts and notations; most definitions and notions used are recalled and refered throughout the paper. If $M$ is a left $C$-comodule, we consider the finite topology on $M^*$ with a basis of neighbourhoods of $0$ consisting of submodules (subspaces) of $M^*$ of the type $X^\perp_M=\{f\in M^*|f\vert_X=0\}$, for finite dimensional subcomodules (subspaces) $X$ of $M$. We write $X^\perp=X^\perp_M$ when there is no danger of confusion. The submodules $X^\perp$ for $X\subset M$ are precisely the closed submodules of $M$. For a subspace  $V$ of $M^*$ we write $V^\perp=\{x\in M|f(x)=0,\,\forall f\in V\}$. We have $X^\perp{}^\perp=X$ for $X\subset M$ and $Y^\perp{}^\perp$ is the closure of $Y\subseteq M^*$. We will frequently use the isomorphisms of $C^*$-modules $(M/X)^*\cong X^\perp$, $M^*/X^\perp\cong X^*$ and $X^\perp/Y^\perp\cong (Y/X)^*$ for subcomodules $X\subseteq Y$ of $M$. Recall, for example, from \cite{IO} that any finitely generated submodule $N$ of $M^*$ is closed in this topology, i.e. $N=X^\perp$, for some subcomodule $X$ of $M$. Also, note that if $X^\perp$ is a closed cofinite left ideal of $C^*$ and $I$ is a left ideal such that $X^\perp\subset I$, then $I$ is closed too: let $f_1,\dots,f_n\in I$ be such that $I=\sum\limits_iC^*f_i+X^\perp$; then $\sum\limits_{i=1}^nC^*f_i=Z^\perp$ is closed, so we get that $I=Z^\perp+X^\perp=(Z\cap X)^\perp$. Also, note that if $C^*/I$ is finite dimensional rational then $I$ is closed, since there is $W^\perp$ closed cofinite ideal such that $W^\perp(C^*/I)=0$, which shows that $W^\perp\subseteq I$.  

We will denote by $J$ the Jacobson radical of $C^*$, $J=Jac(C^*)$; we have that $J=C_0^\perp$, where $C_0$ is the coradical of $C$ (see \cite[Chapter 3]{DNR}). If $C_0\subseteq C_1\subseteq\dots \subseteq C_n$ is the coradical filtration of $C$, then we also have that $(J^n)^\perp=C_n$, $(J^n)^\perp{}^\perp=C_n^\perp$. For a $C$-comodule $M$ we write $L_0M\subseteq L_1M\subseteq \dots \subseteq L_nM\subseteq\dots$ for its Loewy series (coradical filtration), so $L_{n+1}M/L_nM$ is the socle (semisimple sub(co)module of $M/L_nM$). For comodules, one has $M=\bigcup\limits_nL_nM$. One also has that a right comodule $M$ is semisimple if and only if $JM=0$, and $L_nM=M$ if and only if $J^{n+1}M=0$ (e.g. \cite[Lemma 2.2]{ILMS}).


Recall that a coalgebra $(C,\Delta,\varepsilon)$ is called locally finite (\cite{HR74}) if for every two finite dimensional subspaces $X,Y$ of $C$, we have that the "wedge" $X\wedge Y$ is finite dimensional, where $X\wedge Y=\Delta^{-1}(X\otimes C+C\otimes Y)$. Note that by the fundamental theorem of coalgebras, this is equivalent to asking the condition for all finite subcoalgebras $X,Y$ of $C$. Indeed, if it is true for the wedge of finite subcoalgebras, and $X,Y$ are finite dimensional subspaces of $C$, there are finite dimensional subcoalgebras $X',Y'$ such that $X\subseteq X'$ and $Y\subseteq Y'$ and then $X\wedge Y\subseteq X'\wedge Y'$ which is finite dimensional. We also recall that a left (or right) $C$-comodule $X$ is called quasi-finite \cite{Tak2} if and only if $\Hom(S,X)$ for every simple left (right) comodule $S$, or, equivalently, $\Hom(N,X)$ is finite dimensional for every finite dimensional comodule $N$. For a right comodule $(M,\rho)$ subspace $X\subseteq C$ we denote by $cf(X)$ the coalgebra of coefficients of $X$, the smallest subcoalgebra $W$ of $C$ for which $\rho(M)\subseteq M\otimes W$. 

\section{Extensions of finite dimensional rational modules}

One of the important notions connected to the "rational extension problem" will prove to be that of locally finite coalgebras. We thus first note a few interesting caracterizations of locally finite coalgebras.

\begin{lemma}
If $X$ is a right subcomodule of $C$, $W$ a finite dimensional subcoalgebra of $C$, then $(X\wedge W)/X=\sum\limits_{f\in\Hom(P,C/X)}f(P)$, where $P$ is a finite dimensional generator of ${}_{W^*}\Mm=\Mm^W$. In particular, if $S$ is a simple right comodule and $W=cf(S)$, then $X\wedge W=\sum\limits_{f\in\Hom(S,C/X)}f(S)$.
\end{lemma}
\begin{proof}
Let $Y\subset C$ be the subspace such that $Y/X=\sum\limits_{f\in\Hom(P,C/X)}f(S)$, and denote $\pi:C\rightarrow C/X$ the canonical projection. Then for $y\in Y$, write $y=\sum\limits_iy^i$ such that each $(C^*y^i+X)/X$ is a quotient of $P$; since $P\in\Mm^W$ we have $\pi(y^i_1)\otimes y^i_2\in Y/X\otimes W$. Therefore, $\Delta(y^i)=y^i_1\otimes y^i_2\in\pi^{-1}(Y/X\otimes W)=Y\otimes W+X\otimes C$. Indeed, if we write $\Delta(y^i)=\sum\limits_ju_j\otimes a_{ij}+\sum_kv_k\otimes b_{ik}$, where $\{u_j\}$ is a fixed basis of $X$ and $\{v_k\}$ is a fixed basis of $Y$ modulo $X$, then we get $\pi(v_k)\otimes b_{ik}\in Y/X\otimes W$ so $b_{ik}\in W$ since $\pi(v_k)$ is a basis of $Y/X$. Therefore, $y^i\in \Delta^{-1}(X\otimes C+Y\otimes W)\subseteq X\wedge W$, so $Y\subseteq X\wedge W$. Conversely, let $y\in X\wedge W$, so $y_1\otimes y_2=\Delta(y)\in X\otimes C+C\otimes W$. Then $\pi(y_1)\otimes y_2\in C/X\otimes W$, so $C^*\pi(y)$ is canceled by $W^\perp$. Therefore, it has an induced $W^*\cong C^*/W^\perp$-module structure. Hence, there is an epimorphism $(P)^n\rightarrow C^*\pi(y)\rightarrow 0$ as it is finite dimensional, and this shows that $\pi(y)\in Y$. 
\end{proof}

\begin{lemma}\label{l.quasifinite}
The following assertions are equivalent:\\
(i) $C$ is locally finite;\\
(ii) For every finite dimensional right subcomodule $X$ of $C$, $C/X$ is quasifinite.\\
(iii) For every simple right subcomodule $T$ of $C$, $C/T$ is quasifinite.\\
(iv) For every two simple subcoalgebras $U,W$ of $C$, $U\wedge W$ is finite dimensional.\\
(v) The left comodule versions of (ii) and (iii).
\end{lemma}
\begin{proof}
(i)$\Rightarrow$(ii) For simple $S$ and $W=cf(S)$, since $\sum\limits_{f\in\Hom(S,C/X)}f(S)=(X\wedge W)/X$ is finite dimensional, it follows that $\Hom(S,C/X)$ is finite dimensional too. \\
(ii)$\Rightarrow$(i) Use (ii) for a subcoalgebra $X$, which is then also a right subcomodule; take $W$ also a finite dimensional subcoalgebra of $C$ and since $P=W^*$ is a finite dimensional comodule wich generates $\Mm^W$, it follows that $(X\wedge W)/X=\sum\limits_{f\in\Hom(W^*,C/X)}f(W^*)$ is finite dimensional since $\Hom(W^*,C/X)$ is finite dimensional.\\
(ii)$\Rightarrow$(iii) and (i)$\Rightarrow$(iv) are obvious.\\
(iii)$\Rightarrow$(ii) is proved by induction on $length(X)$. For simple $X$ it is obvious. Assume the statement is true for $X$ with $length(X)\leq n$; consider $X$ with $length(X)=n+1$, and let $Y\subset X$ with $X/Y$ simple. Since $C/Y$ is quasifinite, we can write the socle of $C/Y$ as $s(C/Y)=\bigoplus\limits_{i\in I}S_i^{n_i}$, with $S_i$ simple nonisomorphic, $n_i$ finite, and such that $X/Y=S_0$ ($0\in I$). Then there is an embedding $C/Y\hookrightarrow\bigoplus\limits_{i\neq 0}E(S_i)^{n_i}\oplus E(S_0)^{n_0}$ with finite $n_i$'s, which can be extended to an essential embedding $C/X\cong (C/Y)/(X/Y)\hookrightarrow H=E(S_0)/S_0\oplus E(S_0)^{n_0-1}\oplus \bigoplus\limits_{i\neq 0}E(S_i)^{n_i}$. Now, since for each simple $S$, $\Hom(S,E(S_0)/S_0)$ is finite dimensional, then $\Hom(S,H)=\Hom(S,E(S_0)/S_0)\oplus \Hom(S,E(S_i)^{n_i-\delta_{i,0}})$ is obviously finite dimensional. It follows that $\Hom(S,C/X)$ is finite dimensional, and the proof is finished.\\
(iv)$\Rightarrow$(iii) If $T,L$ are a simple right subcomodules of $C$, let $U=cf(T)$, $W=cf(L)$; then $T\wedge W\subseteq U\wedge W$ is finite dimensional. Therefore, by the previous Lemma, we get $\sum\limits_{f\in\Hom(L,C/T)}f(L)=(T\wedge W)/T$ is finite dimensional, and so $\Hom(L,C/T)$ is finite dimensional. 
\end{proof}

We prove now the connection mentioned at the beginning of this section:

\begin{proposition}\label{p.ratlf}
Let $C$ be a coalgebra such that left rational $C^*$-modules are closed under extensions. Then $C$ is locally finite.
\end{proposition}
\begin{proof}
Assume $C$ is not locally finite. Then there are simple left comodules $S,L$ such that $\Hom(L,C/S)$ is infinite dimensional. Let $X\subseteq C$ be such that $X/S\cong \bigoplus\limits_{n\in\NN}L$, and let $X_n/S=\bigoplus\limits_{k\in\NN\setminus\{n\}}L$. Then $X/X_n\cong L$ and  $\bigcap\limits_nX_n=S$. Let $I=\sum\limits_nX_n^\perp\subseteq S^\perp$; note that $I^\perp=\bigcap\limits_nX_n^\perp{}^\perp=\bigcap\limits_nX_n=S$, and $X^\perp\subseteq X_n^\perp$. Note that $T^\perp/X^\perp\cong (X/T)^*\cong \prod\limits_nL^*$, which is a rational module, since it is canceled by $cf(L)^\perp$; moreover, it becomes a $cf(L)^*$-module, with $cf(L)$ a finite dimensional simple coalgebra, so it is semisimple. Therefore, every quotient of $T^\perp/X^\perp$ is rational and semisimple. Obviously, $I\neq T^\perp$, since $I/X^\perp$ is countable dimensional while $T^\perp/X^\perp$ is uncountable dimensional. Therefore, since $T^\perp/X^\perp$ is semisimple rational, we can find $I\subset K\subsetneq T^\perp$ such that $T^\perp/K$ is rational simple. Since $C^*/T^\perp$ is rational, the hypothesis shows that $C^*/K$ is rational. Therefore, $K$ is closed by the remarks in the introduction: $K=Y^\perp$. But then $Y=K^\perp\subset I^\perp=S$, so $K=Y^\perp=S^\perp$, a contradiction.
\end{proof}

We note that in \cite{NT3} a comodule $M$ such that every quotient of $M$ is quasifinite was called strictly quasi-finite. Lemma \ref{l.quasifinite} provides an interesting connection between the notion of locally finite coalgebra and a slightly weaker version of "strictly quasifinite coalgebra", where quotients only by finite dimensional comodules are quasifinite. In particular, we have

\begin{corollary}
A strictly quasifinite coalgebra is locally finite.
\end{corollary}


We note however that the converse of this corollary does not hold. We use an example belonging to a family which provides many examples and counterexamples in coalgebras (see \cite{IInt}). These are coalgebras which are subcoalgebras of the full quiver coalgebra of a quiver and have a basis of paths; such coalgebras were called path subcoalgebras in \cite{DIN}, and have also been called monomial in literature.

\begin{example}\label{e.1}
Consider the following quiver:
$$\xygraph{ !{0;(.77,-.77):0}
!~:{@{-}|@{>}}
a(:@/^.7pc/[dl(.7)]{b_1}^{x_1}:@/^.7pc/^(.4){y_1}"a",:@/^.7pc/[ul(.7)]{b_2}^(.6){x_2}:@/^.7pc/^{y_2}"a",
[ur(.7)]{\dots},:@/^.7pc/[r]{b_n}^{x_n}:@/^.7pc/^{y_n}"a",[dr(.7)]{\dots})}
$$
and let $C$ be the $\KK$-coalgebra with basis $\{a,b_n,x_n,y_n,p_n|n\in\NN\}$ where $p_n$ is the path $p_n=x_ny_n$, as a subcoalgebra of the full path coalgebra of the above quiver. We have thus formulas for $\Delta$ and $\varepsilon$:
\begin{eqnarray*}
\Delta(a) & = & a\otimes a\\
\Delta(b_n) & = & b_n\otimes b_n\\ 
\Delta(x_n) & = & a\otimes x_n +x_n\otimes b_n \\ 
\Delta(y_n) & = & b_n\otimes y_n+y_n\otimes a \\ 
\Delta(p_n) & = & a\otimes p_n+x_n\otimes y_n+p_n\otimes b_n
\end{eqnarray*}
$\varepsilon(a)=\varepsilon(b_n)=1$; $\varepsilon(x_n)=\varepsilon(y_n)=\varepsilon(p_n)=0$. As in \cite{DIN}, we have $C_0=\KK\{a,b_n|n\in\NN\}$ (this is the $\KK$-span), $C_1=\KK\{a,b_n,x_n,y_n|n\in\NN\}$, $C=C_2$. We note then that $C/C_1\cong \bigoplus\limits_{\NN}\KK\{a\}$ as left comodules, since in $C/C_1$, the comultiplication maps $\overline{p_n}\longrightarrow a\otimes\overline{p_n}$. This shows that $C$ is not strictly quasifinite. However, we note that $C$ is locally finite. Let $X,Y$ be finite subspaces of $C$; then there are finite dimensional subcoalgebras $U,W$ which have bases of paths and such that $X\subseteq U$ and $Y\subseteq W$. For example, one can take $U$ to be the span of all paths which occur in elements in $X$ and their respective subpaths. Let $V_n=\KK\{a,b_k,x_k,y_k,p_k|k\leq n\}$. Then $V_n$ is a subcoalgebra of $C$ and $\bigcup\limits_nV_n=C$. Since $U,W$ are finite dimensional, it follows that there is some $n$ such that $U,W\subseteq V_n$. But now it is easy to see that $V_n\wedge V_n=V_n$, so $X\wedge Y\subseteq U\wedge W\subseteq V_n\wedge V_n=V_n$. This shows that $C$ is locally finite.
\end{example}

\begin{proposition}
Assume the finite dimensional left rational $C^*$-modules are closed under extensions. Let $0\rightarrow M\rightarrow P\rightarrow N\rightarrow 0$ be an exact sequence of $A$-modules such that $M,N$ are rational, $N$ is finite dimensional and $P$ is cyclic.  Then $M/JM$ is finite dimensional.
\end{proposition}
\begin{proof}
Note that if we quotient out by $JM$, we get an exact sequence $0\rightarrow M/JM\rightarrow P/JM\rightarrow N\rightarrow 0$ with the same properties, so we may in fact assume that $M$ is semisimple, and show that then $M$ is finite dimensional.\\
Let us first show that in $M$ there are only finitely many isomorphism types of simple comodules. Indeed, assume otherwise, and then again after taking another quotient we may assume we have an exact sequence $0\rightarrow \bigoplus\limits_{n
\in\NN}S_n\rightarrow P'\rightarrow N\rightarrow 0$, with $S_n$ nonisomorphic simple comodules. Also, $P$' is cyclic so there is an epimorphism $\pi:C^*\rightarrow P'$. Let $I=\pi^{-1}(\bigoplus\limits_nS_n)$; then since obviously $C^*/I\cong N$ is finite dimensional rational, we have that $I$ is closed: $I=X^\perp$, with $X$ a finite dimensional left subcomodule of $C$. Also, if $M_n=\bigoplus\limits_{k\neq n}S_k$, then we have an exact sequence $0\rightarrow S_n\rightarrow P'/M_n\rightarrow N\rightarrow 0$, and as the finite dimensional rational left $C^*$-modules are closed under extensions, we see that $P'/M_n$ is rational. Therefore, $\pi^{-1}(M_n)=X_n^\perp$, $X\subset X_n\subseteq C$. Note that $\bigcap\limits_{k\in F}M_k\not\subseteq M_n$ for any finite set $F\subset \NN$ with $n\notin F$. This shows that $(\sum\limits_{k\in F}X_k)^\perp=\bigcap\limits_{k\in F}X_k^\perp\not\subseteq M_n=X_n^\perp$, so $X_n\not\subseteq\sum\limits_{k\in F}X_k$. Also, $(X_n/X)^*\cong X^\perp/X_n^\perp\cong S_n$, so $X_n/X$ is simple. Let $Y=\sum\limits_{n\in \NN}X_n$; the previous considerations show that $Y/X\cong\bigoplus\limits_nS_n^*$. Also, $$\ker(\pi)=\pi^{-1}(0)=\pi^{-1}(\bigcap\limits_{n\in\NN}M_n)=\bigcap\limits_{n\in\NN}\pi^{-1}(M_n)=\bigcap\limits_{n\in\NN}X_n^\perp\supseteq Y^\perp.$$
We thus have an epimorphism $I/Y^\perp\rightarrow I/\ker{\pi}\cong \bigoplus\limits_{n\in\NN}S_n$, so an epimorphism $X^\perp/Y^\perp\rightarrow \bigoplus\limits_nS_n$. But $X^\perp/Y^\perp=(Y/X)^*=(\bigoplus\limits_{n\in\NN}S_n^*)^*=\prod\limits_{n\in\NN}S_n$, so we have obtained an epimorphism of left $C^*$-modules.
$$\prod\limits_{n\in\NN}S_n\rightarrow \bigoplus\limits_{n\in\NN}S_n\rightarrow 0$$
But now, since $\bigoplus\limits_n{S_n^*}$ is a direct sum of nonisomorphic simple rational left comodules, we have an embedding $\bigoplus\limits_{n\in\NN}S_n^*\hookrightarrow C_0\subset C$, and therefore, dualizing we get an epimorphism $C^*\rightarrow \prod\limits_{n\in\NN}S_n$. Combining with the above, we obtain an epimorphism $C^*\rightarrow \bigoplus\limits_{n\in \NN}S_n$; but $\bigoplus\limits_{n\in\NN}S_n$ is not finitely generated, and this is a contradiction.\\
Now, let us show that it is not possible to have infinitely many copies of the same comodule $S$ as summands of $M$. Assume otherwise, and keep the above notations, only now we will have $S_n\cong S$, for some simple right comodule $S$. As above, we obtain $Y/X\cong\bigoplus\limits_{n\in\NN}S$, and since $X$ is finite dimensional, this contradicts the hypothesis that $C$ is locally finite, which follows by Proposition \ref{p.ratlf}. \\
This ends our proof.
\end{proof}

\begin{proposition}\label{p.2}
Assume the same hypotheses (and notations) of the previous propositions hold. Then the sequence $M\supseteq JM\supseteq J^2M\supseteq \dots\supseteq J^nM\supseteq\dots $ eventually terminates: $J^nM=J^{n+1}M=\dots$, and $M/J^nM$ is finite dimensional. 
\end{proposition}
\begin{proof}
By the previous proposition, we have $M/JM$ is finite dimensional. Therefore, since the finite dimensional rational are closed under extensions, $P/JM$ is rational too since it fits into the exact sequence $0\rightarrow M/JM\rightarrow P/JM\rightarrow P/M\rightarrow 0$. We can again apply the previous proposition for the sequence $0\rightarrow JM\rightarrow P\rightarrow P/JM\rightarrow 0$ and obtain that $M/J^2M$ is finite dimensional, and inductively, we obtain $M/J^nM$ is finite dimensional and also that $P/J^nM$ is rational. Let again $\pi:C^*\rightarrow P$ be an epimorphism. As before, $I=\pi^{-1}(M)$ is closed, so $I=X^\perp$, with $X$ a left subcomodule of $C$. Similarly, since $C^*/\pi^{-1}(J^nM)\cong P/J^nM$ is rational, so again $\pi^{-1}(J^nM)=X_n^\perp$, for a left finite dimensional subcomodule $X_n\subset C$. Obviously $X_n^\perp\supseteq X_{n+1}^\perp$, so $X_n\subseteq X_{n+1}$, and it suffices to show that $X_n=X_{n+1}=\dots$ from some $n$ onward. Let $Y=\bigcup\limits_{n\in\NN}X_n$; since $\ker(\pi)\subseteq X_n^\perp$, $\ker(\pi)\subseteq \bigcap\limits_{n\in\NN}X_n^\perp=(\sum\limits_{n\in\NN}X_n)^\perp=Y^\perp$. Therefore, there is an epimorphism induced by $\pi$, making the diagram commutative:
$$\xymatrix{
X^\perp\ar[r]^\pi\ar[d]^p & M\ar[dl]^{\overline{p}} \\
X^\perp/Y^\perp
}$$
We now show that the socle of $Y/X$ is $X_1/X$. For this, let $Z/X$ be a simple comodule with $Z\subset Y$; this means that $(Z/X)\cdot J=0$ or, equivalently $J(Z/X)^*=J\cdot (X^\perp/Z^\perp)=0$. But $J\cdot (Y/X)^*=J\overline{p}(M)=\overline{p}(\pi(\pi^{-1}(JM)))=\overline{p}\pi(X_1^\perp)=p(X_1^\perp)=X_1^\perp/Y^\perp$. Therefore
\begin{eqnarray*}
J\cdot(X^\perp/Z^\perp) & = & J\cdot(\frac{X^\perp/Y^\perp}{Z^\perp/Y^\perp})=\frac{J\cdot(X^\perp/Y^\perp)+Z^\perp/Y^\perp}{Z^\perp/Y^\perp}\\
& = & \frac{X_1^\perp/Y^\perp+Z^\perp/Y^\perp}{Z^\perp/Y^\perp}=0
\end{eqnarray*}
This shows that $X_1^\perp/Y^\perp\subseteq Z^\perp/Y^\perp$ and so $Z\subseteq X_1$. Now, since the socle of $Y/X$ is $X_1/X$ and is finite dimensional, there is an embedding $Y/X\hookrightarrow C^n$ for some $n$ and so, by duality, we get an epimorphism $(C^*)^n\rightarrow (Y/X)^*\rightarrow 0$, so $(Y/X)^*$ is finitely generated. Since it is rational (a quotient of $M$), it has to be finite dimensional. Therefore, the sequence $X_n\subseteq X_{n+1}\subseteq X_{n+2}\subseteq\dots$ must terminate. This ends the proof.
\end{proof}

We note that by \cite[Lemma 2.10]{CNO} we have that the set of closed cofinite left ideals of $C^*$  is closed under products if $C$ is locally finite and $X^\perp Y^\perp=(X\wedge Y)^\perp$ for all left finite dimensional subcomodules $X,Y$ of $C$. We note that with the same proof, we have that the set of two-sided closed cofinite ideals of $C^*$ is closed under products if and only if $C$ is locally finite and $U^\perp W^\perp=(U\wedge W)^\perp$ for all finite dimensional subcoalgebras $U,W$ of $C$. We see that these two conditions are in fact equivalent:

\begin{proposition}\label{p.3}
The following assertions are equivalent for a coalgebra $C$:\\
(i) The finite dimensional rational $C$-comodules are closed under extensions;\\
(ii) The set of left closed cofinite ideals (equivalently, open ideals) of $C^*$ is closed under products (of ideals), i.e. $X^\perp Y^\perp$ is closed cofinite whenever $X,Y$ are finite dimensional left subcomodules of $C$.\\
(iii) The set of two-sided cofinite closed ideals of $C^*$ is closed under products.\\
(iv) The right hand side version of (ii).
\end{proposition}
\begin{proof}
(iii)$\Rightarrow$(ii) If $X,Y$ are finite dimensional left subcomodules of $C$ so that $X^\perp$, $Y^\perp$ are closed cofinite left ideals in $C^*$, let $U=cf(X)$, $W=cf(Y)$ which are finite dimensional subcoalgebras of $C$; then $U^\perp W^\perp\subseteq X^\perp Y^\perp$. But by hypothesis $U^\perp W^\perp=(U^\perp W^\perp)^\perp{}^\perp=(U\wedge W)^\perp$ (here we can also use \cite[Lemma 2.5.7]{DNR}) is closed and cofinite, and since $X^\perp Y^\perp\supseteq (U\wedge W)^\perp$ it follows that $X^\perp Y^\perp$ is closed and cofinite.\\
(ii)$\Rightarrow$(iii) is obvious.\\
(iii)$\Rightarrow$(i) Let $0\rightarrow M'\rightarrow M\rightarrow M''\rightarrow 0$  be an exact sequence of $C^*$-modules with $M',M''$ finite dimensional rational. Let $U=cf(M')$, $W=cf(M'')$. Then $U^\perp\cdot M'=0$, $W^\perp\cdot M''$ (in fact, $U=ann_{C^*}(M')^\perp$ by \cite[Proposition 2.5.3]{DNR}), and then, for $x\in M$ it follows that $W^\perp x\in M'$ and thus $U^\perp W^\perp x=0$, so $U^\perp W^\perp\subset ann_{C^*}(M)$. But by the hypothesis of (iii), as before we have $U^\perp W^\perp=(U\wedge W)^\perp$ and this is closed cofinite. Hence since each $x\in M$, is canceled by a closed cofinite ideal $(U\wedge W)^\perp$, it follows that $C^*x$ is rational, so $M$ is rational.\\
(i)$\Rightarrow$(iii) Let $U,W$ be finite dimensional subcoalgebras of $C$. Consider the exact sequence 
\begin{equation}\label{eq1}
0\rightarrow W^\perp/U^\perp W^\perp\rightarrow C^*/U^\perp W^\perp\rightarrow C^*/W^\perp\rightarrow 0
\end{equation}
Note that $M'=W^\perp/U^\perp W^\perp$ is rational since it is canceled by the cofinite closed ideal $U^\perp$ (and we can argue as above in (iii)$\Rightarrow$(i)). Moreover, there is some $n$ such that $U\subseteq C_{n-1}$, and then $U\subset {(J^n)}^\perp=C_{n-1}$ (use again \cite[Lemma 2.5.7]{DNR}). Then $J^n\subseteq C_{n-1}^\perp\subseteq U^\perp$ so $J^n\cdot M'=0$. We now note that we are under the assumptions of Proposition \ref{p.2}: $C^*/U^\perp W^\perp$ is cyclic, $C^*/W^\perp\cong W^*$ is finite dimensional rational and $M'=W^\perp/U^\perp W^\perp$ is rational. Therefore, it follows that $J^kM'=J^{k+1}M'=\dots$ from some $k$. But also for large $k$, $J^kM'=0$. Moreover, $M'/J^kM'$ is also finite dimensional, and so it follows that $M'$ is finite dimensional. It follows that the exact sequence in (\ref{eq1}) is a sequence of finite dimensional $C^*$-modules, with rational "ends", and by the hypothesis of (i) it follows that $M=C^*/U^\perp W^\perp$ is finite dimensional rational. Therefore, if $H=cf(M)$, then $H$ is finite dimensional and we have $H^\perp M=0$, and so $H^\perp\subseteq U^\perp V^\perp$. It then follows by the initial remarks that $U^\perp V^\perp$ is closed and cofinite.\\
Equivalence with (iv) follows by the symmetry of (i) and (iii).
\end{proof}

\section{Connection to coreflexive coalgebras}

We note now a connection with another important notion in coalgebra theory, that of coreflexive coalgebras. Recall from \cite{Tf} that a coalgebra is coreflexive if the natural map $C\rightarrow (C^*)^o$ is surjective (so an isomorphism). By the results of \cite{HR74}, there is a tight connection of this to rational modules: $C$ is coreflexive if and only if every finite dimensional left $C^*$-module is rational; by symmetry, this is equivalent to every finite dimensional right $C^*$-module being rational. Therefore we have:

\begin{proposition}\label{p.reflexive}
A coalgebra $C$ is coreflexive if and only if $C_0$ is coreflexive and the finite dimensional rational left (or, equivalently, right) $C^*$-modules are closed under extensions.
\end{proposition}
\begin{proof}
If $C$ is coreflexive, then $C_0$ is coreflexive as a subcoalgebra of $C$ (by \cite[3.1.4]{HR74}); obviously the finite dimensional $C^*$-modules are closed under extensions since if 
\begin{equation}\label{eq2}
0\rightarrow M'\rightarrow M\rightarrow M''\rightarrow 0
\end{equation}
is an exact sequence with finite rational $M',M''$, then $M$ is finite dimensional so it is rational by the hypothesis. Conversely, we proceed as in the proof of \cite[Theorem 7.1]{DIN} by induction on the length (or dimension) of the finite dimensional module $M$ to show that finite dimensional $C^*$-modules are rational: it is true for the the simple ones since $C_0$ is coreflexive, and in general, take a sequence as in (\ref{eq2}) for some proper subcomodule $M'$ of $M$ and apply the induction hypothesis on the $C^*$-modules $M',M''$ of length smaller than $M$ to get that they are rational. Therefore, $M$ is rational as an extension of $M''$ by $M'$.
\end{proof}

Applying proposition \ref{p.3} we get:

\begin{corollary}
$C$ is coreflexive if and only if $C_0$ is coreflexive and the topology of closed cofinite (open) ideals of $C^*$ is closed under products.
\end{corollary}

We also note the following: by the results of \cite{Rad} (see \cite[3.12]{Rad}), if $\KK$ is infinite, we have that a cosemisimple coalgebra $C=\coprod\limits_{i\in I}C_i$ with $C_i$ simple coalgebras is coreflexive if and only if $\coprod\limits_{i\in I}\KK=\KK^{(I)}$ is coreflexive. On the other hand, by \cite[Section 3.7]{HR74}, the coalgebra $\KK^{(I)}$ is coreflexive for every set $I$ whose cardinality is nonmeasurable (or if ${\rm card}(I)<{\rm card}(\KK)$). A cardinal $X$ is called nonmeasurable if every Ulam ultrafilter on $X$ (that is, an ultrafilter which is closed under countable intersections) is principal (i.e. equal to the set of all subsets of $X$ containing some $x_0\in X$). The class of nonmeasurable sets is closed under "usual" constructions, such as subsets, unions, products, power set, and contains the countable set, and in fact there is no known example of set wich is measurable (i.e. not nonmeasurable). Hence, it is reasonable to espect that $C_0$ is coreflexive, for any coalgebra $C$ over an infinite field $\KK$ (for finite fields, \cite[3.12]{Rad} provides a simple test). Hence, the above corollary shows that in general, if the set of simple comodules of $C$ is any set that we might "reasonably" expect, to test coreflexivity is equivalent to testing that the finite topology of ideals of $C^*$ is closed under products, equivalently, $C$ is locally finite and $U^\perp W^\perp=(U\wedge W)^\perp$.

\section{Rational Extensions and the homological Ext}

\begin{lemma}\label{l.1}
Assume that $C$ is a coalgebra such that if $0\rightarrow M'\rightarrow M\rightarrow M''\rightarrow 0$ is any exact sequence of left $C^*$-modules with $M',M''$ rational and $M''$ cyclic, then $M$ is rational. Then the left rational $C^*$-modules are closed under extensions.
\end{lemma}
\begin{proof}
Let 
$$0\rightarrow M'\rightarrow M\stackrel{\pi}{\rightarrow} M''\rightarrow 0$$
be an exact sequence with $M',M''$ rational. Write $M=M'+\sum\limits_iC^*x_i$, so that the images $\pi(x_i)$ of $x_i$ in $M''$ generate $M''$. Since $M''$ is rational, $C^*\pi(x_i)$ is finite dimensional, and so we have exact sequences $0\rightarrow M'\rightarrow M'+C^*x_i\rightarrow C^*\pi(x_i)\rightarrow 0$, which shows that $M'+C^*x_i$ is rational. Therefore, $M=\sum\limits_i(M'+C^*x_i)$ is rational.
\end{proof}

\begin{proposition}\label{p.5}
Let $0\rightarrow M'\rightarrow M\rightarrow M''\rightarrow 0$ be an exact sequence with $M',M''$ rational, and such that $M'$ has finite Loewy length (finite coradical filtration). If finite dimensional rational left $C^*$-modules are closed under extensions, then $M$ is rational. 
\end{proposition}
\begin{proof}
The previous Lemma shows that we may assume $M''$ is cyclic rational and finite dimensional. By \cite[Lemma 2.2]{ILMS} we have that $J^nM'=0$ for some $n$ since $M'$ has finite Loewy length. Proposition \ref{p.2} now shows that $M'=M'/J^nM'$ is finite dimensional. Therefore, $M$ is finite dimensional since finite dimensional rationals are assumed closed under extensions.
\end{proof}

\begin{proposition}\label{p.6}
Assume $Rat({}_{C^*}\Mm)$ is not localizing, but finite dimensional rational modules are closed under extensions. Then there is an exact sequence $0\rightarrow M\rightarrow P\rightarrow S\rightarrow 0$ and such that \\
(i) $P$ is cyclic and non-rational.\\
(ii) $JM=M$ and $M$ has simple socle (and infinite Loewy length).\\
(iii) $S$ is simple finite dimensional rational.
\end{proposition}
\begin{proof}
By Lemma \ref{l.1} we can find such a sequence with $P$ cyclic and $S$ finite dimensional. Let us consider such a sequence with $S$ of minimal length (or dimension). Then we see that $S$ is simple. Indeed, if $S$ is not, then let $x\in P\setminus M$ such that $L=(C^*x+M)/M$ is simple, $L\subsetneq S$. If $C^*x$ is rational, then we get an exact sequence $0\rightarrow M+C^*x/C^*x\rightarrow P/C^*x\rightarrow P/(M+C^*x)\rightarrow 0$, which has the properties: $P/(M+C^*x)$ has positive length smaller than $S$, $P/C^*x$ is cyclic and not rational (if it were rational, $P$ would be rational since then $P/C^*x$ and $C^*x$ would be finite dimensional rational). This contradicts the minimality choice. Therefore, $C^*x$ is not rational, and we get an exact sequence $0\rightarrow M\cap C^*x\rightarrow C^*x\rightarrow C^*x/(C^*x\cap M)\cong L\rightarrow 0$, with $M\cap C^*x$ rational and $L$ simple. Since the initial $S$ was of minimal possible length of all sequences with this feature, $S$ is simple.\\
Now, note that $M$ must have infinite Loewy length, since otherwise $P$ would be  rational by Proposition \ref{p.5}. Also, by Propsition \ref{p.2}, we can find $n$ such that $J^nM=J^{n+1}M$ and $M/J^nM$ is finite dimensional. Since $M$ is rational, we can find a finite dimensional subcomodule $N$ of $M$ such that $J^nM+N=M$. Let $M'=M/N$, $P'=P/N$, so that we have an exact sequence $0\rightarrow M'\rightarrow P'\rightarrow S\rightarrow 0$. Note that $JM'=J(J^nM+N)/N=(J^{n+1}M+JN+N)/N=(J^nM+N)/N=M'$. Let $L$ be a simple subcomodule of $M'$ and $X$ be a maximal subcomodule of $M'$ such that $L\cap X=0$ (such a subcomodule can be found, for example, by Zorn's Lemma). It is then not difficult to see that $L=L+X/X \hookrightarrow M'/X$ is an essential subcomodule of $M''=M'/X$. Moreover, $JM''=J(M'/X)=(JM'+X)/X=M'/X=M''$. This also shows that $M''$ has infinite Loewy length, since $M''\neq 0$ (for example, again by \cite[Lemma 2.2]{ILMS}). Then, if $P''=P'/X$, we have $P''/M''\cong P'/M'\cong P/M=S$, and so the exact sequence $0\rightarrow M''\rightarrow P''\rightarrow S\rightarrow 0$ has the required properties: $JM''=M''$, $M''$ has simple socle and infinite Loewy length, $S$ is simple and $P''$ is cyclic and nonrational (if $P''$ is rational, it is finite dimensional and in this case so is $M''$, which is not possible).
\end{proof}

\begin{proposition}\label{p.7}
Let $0\rightarrow M\rightarrow P\rightarrow S\rightarrow 0$ be a short exact sequence of left $C^*$-modules with $M,S$ rational modules, and $S$ a simple module. Then there is an extension $0\rightarrow E(M)\rightarrow \overline{P}\rightarrow S\rightarrow 0$, where $E(M)$ is the injective hull of $M$ as right. Moreover 
if $P$ is not rational, then $\overline{P}$ is not rational either, and the last sequence is not split.
\end{proposition}
\begin{proof}
We take $\overline{P}$ to be the pushout in the category of left $C^*$-modules of the following diagram:
$$\xymatrix{
& E(M)\ar@{^{(}.>}[dr] &\\
M\ar@{^{(}->}[ur]\ar@{^{(}->}[dr] & & \overline{P} \\
& P\ar@{^{(}.>}[ur] &
}$$
By the properties (or the construction) of the push-out, we have $E(M)+P=\overline{P}$, $E(M)\cap P=M$; this shows that $\overline{P}/E(M)=(E(M)+P)/E(M)=P/P\cap E(M)=P/M\cong S$ and this gives us the required extension. 
If $P$ is not rational, then $\overline{P}$ is not rational either, as it contains $P$. Obviously, the sequence $0\rightarrow E(M)\rightarrow \overline{P}\rightarrow S\rightarrow 0$ is not split, since otherwise, $\overline{P}$ would be rational. 
\end{proof}

We note a few equivalent interpretations of the condition in the above proposition. 

\begin{proposition}\label{p.8}
Let $S$ be a left rational $C^*$-module, $T=S^*$, and $E$ an injective right $C$-comodule. The following assertions are equivalent:\\
(i) $\Ext^1(S,E)=0$;\\
(i') Any sequence $0\rightarrow E\rightarrow \overline{P}\rightarrow S\rightarrow 0$ splits.\\
(i'') $E$ is injective in the localizing subcategory (i.e. closed under extensions) of ${}_{C^*}\Mm$ generated by $Rat({}_{C^*}\Mm)$.\\
(ii) For any $f\in\Hom_{C^*}(T^\perp,E)$, ${\rm Im}(f)$ is finite dimensional.\\
(iii) If $M=T^\perp_{E(T)}$ is the maximal submodule of $E(T)^*$ (which is unique by \cite[Lemma 1.4]{I2}), any $f\in\Hom_{C^*}(M,E)$ has finite dimensional image.
\end{proposition}
\begin{proof}
The equivalence of (i)-(i'') is obvious.\\
If we write $C=E(T)\oplus H$, we note that $C^*=E(T)^*\oplus H^*$ so then $T^\perp=M\oplus H^*$. Since $H^*$ is cyclic (as a direct summand of $C^*$), we get that morphisms $f\in\Hom_{C^*}(H^*,E)$ have finite images. This shows the equivalence of (ii) and (iii).\\
(i)$\Rightarrow$(ii) The condition yields an exact sequence $0\rightarrow \Hom_{C^*}(S,E)\rightarrow\Hom_{C^*}(C^*,E)\rightarrow\Hom_{C^*}(T^\perp,E)\rightarrow \Ext^1(S,E)=0$, therefore any $f$ in the following diagram extends to some $g$:
$$\xymatrix{
0\ar[r] & \ar[d]_f T^\perp \ar[r] & C^*\ar@{..>}[dl]^g \\
& E &
}$$
Therefore, since ${\rm Im}(g)$ is cyclic rational, it is finite dimensional, and ${\rm Im}(f)\subseteq {\rm Im}(g)$.\\
(ii)$\Rightarrow$(i) Let $f\in\Hom_{C^*}(T^\perp,C^*)$, $K=\ker(f)$, $F={\rm Im}(f)$; we have the commutative diagram:
$$\xymatrix{
0\ar[r] & \ar@{->>}[d]_{\overline{f}} T^\perp \ar[r]^\sigma & C^*\ar@{->>}[d]^{\overline{g}} \\
0\ar[r] & F\ar[d]_{u}\ar[r]^{\overline{\sigma}} & C^*/\sigma(K)\ar@{..>}[dl]^{v} \\
& E &
}$$
where $f=u\circ \overline{f}$, $\overline{g}$ is the canonical projection and $\overline{\sigma}$ is induced by $\sigma$. Since $C^*/T^\perp=T^*$ and $T^\perp/K$ are finite dimensional, the second row of the diagram consists of finite dimensional comodules; since $E$ is injective, the diagram extends with some $v$ such that $v\circ \overline{\sigma}=u$ (e.g. by \cite[Theorem 2.4.17]{DNR}). Hence, $g=v\circ\overline{g}$ extends $f$. This gives an exact sequence 
$$0\rightarrow \Hom_{C^*}(S,E)\rightarrow\Hom_{C^*}(C^*,E)\stackrel{\sigma^*}{\rightarrow}\Hom_{C^*}(T^\perp,E)\stackrel{0}{\rightarrow} \Ext^1(S,E)\stackrel{0}{\rightarrow} \Ext^1(S,C^*)=0$$
Since $\sigma^*$ is surjective, it is standard to see that the above sequencewe yields $\Ext^1(S,E)=0$.
\end{proof}

\begin{corollary}
Let $C$ be a coalgebra and assume finite dimensional left rational modules are closed under extensions. Then the following are equivalent:\\
(i) $Rat({}_{C^*}\Mm)$ is closed under extensions (i.e. $Rat$ is a torsion functor).\\
(ii) $\Ext^1(S,E)=0$ for every simple right $C$-comodule $S$ and every injective indecomposable right $C$-comodule $E$.\\
(iii) $\Ext^1(S,C)=0$, for all simple right $C$-comodules $S$. \\
(iii)' $\Ext^1(C_0,C)=0$ as left $C^*$-modules.\\
(iv) For every simple left $C$-comodule $T$, every injective indecomposable right comodule $E$ and any $f\in\Hom_{C^*}(T^\perp,E)$, $\im(f)$ is finite dimensional.\\
(v) There is no exact sequence of left $C^*$-modules $0\rightarrow M\rightarrow P\rightarrow S\rightarrow 0$ with $M$ rational with simple socle and $JM=M$, $S$ simple rational and $P$ cyclic.
\end{corollary}
\begin{proof}
(v)$\Rightarrow$(i) follows from Proposition \ref{p.6}\\ 
(i)$\Rightarrow$(ii) and (iii) follows since if $0\rightarrow E\rightarrow \overline{P}\rightarrow S\rightarrow 0$ is an exact sequence, (i) implies that $\overline{P}$ is rational and so the sequence splits since $E$ is an injective rational comodule. \\
(ii)$\Rightarrow$(v) follows by Proposition \ref{p.7}.\\
(iv)$\Leftrightarrow$(ii) is contained in Proposition \ref{p.8}\\
(iii)$\Rightarrow$(ii) is obvious since $C=E\oplus H$ for some $H$. \\
(iii)$\Leftrightarrow$(iii)' is obvious.
\end{proof}

We can now proceed with the main characterization of the "rationals closed under extension" property.

\begin{theorem}\label{t.1}
Let $C$ be a coalgebra. Then the left rational $C^*$-modules are closed under extensions if and only if the following two conditions hold:\\
(i) products of closed cofinite ideals are closed; (equivalently, $C$ is locally finite and for any two finite dimensional subcoalgebras $V,W$ of $C$, $V^\perp W^\perp=(V\wedge W)^\perp$.)\\
(ii) $\Ext^1(S,E)=0$, for every simple right $C$-comodule $S$ and injective right $C$-comodule $E$ (equivalently, $\Ext^1(S,C)=0$, for all such $S$, or $\Ext^1(C_0,C)=0$ as left $C^*$-modules). 
\end{theorem}

We note that $\Ext^1(S,E)$ (or $\Ext^1(C_0,C)$) can be computed from the long exact sequence of homology:

$$0\rightarrow \Hom_{C^*}(S,E)\rightarrow\Hom_{C^*}(C^*,E)\stackrel{}{\rightarrow}\Hom_{C^*}(T^\perp,E)\stackrel{}{\rightarrow} \Ext^1(S,E)\stackrel{}{\rightarrow} \Ext^1(S,C^*)=0$$

where $T=S^*$. So $\Ext^1(S,E)=\frac{\Hom(T^\perp,E)}{\hookrightarrow \Hom(C^*,E)}$ with ${\hookrightarrow \Hom(C^*,E)}$ representing the image of $\Hom(C^*,E)$ in $\Hom(T^\perp,E)$; similarly, $\Ext^1(C_0,C)=\frac{\Hom(C_0^\perp,C)}{\hookrightarrow\Hom(C^*,C)}$.

\section{A general sufficient condition and applications}

We give a quite general sufficient condition under we have that the rational functor is a torsion functor. In fact, these will be situations in which closure of finite dimensional rationals under extensions is enough to have closure under extensions of all rational modules. A right comodule $M$ is called finitely cogenerated if it embedds in a finite direct sum of copies of $C$. 

First, let us note:

\begin{proposition}
Assume any injective indecomposable right comodule $E$ has finite Loewy length. If open (i.e. closed cofinite) ideals of $C^*$ are closed under products, then $Rat({}_{C^*}\Mm)$ is closed under extensions.
\end{proposition}
\begin{proof}
If rational modules are not closed under extensions, there is an exact sequence $0\rightarrow M\rightarrow P\rightarrow S\rightarrow 0$ as in Proposition \ref{p.6}, with $M$ with simple socle and of infinite Loewy length; but then $E(M)$ is indecomposable of infinite Loewy length, a contradiction.
\end{proof}

\begin{theorem}\label{t.2}
Let $C$ be a coalgebra such that ideals of $C^*$ are closed under products. Assume that for each left indecomposable injective comodule $E(T)$ there is $0\neq X\subseteq E(T)$ such that $E(T)/X$ is finitely cogenerated and $X$ has finite Loewy length. Then left rational $C^*$-modules are closed under extensions.
\end{theorem}
\begin{proof}
Assume the contrary, and consider an exact sequence of left $C^*$-modules as provided by Proposition \ref{p.6}:
$$0\rightarrow M\rightarrow P\stackrel{\pi}{\rightarrow} S\rightarrow 0$$
Let $T=S^*$; since $P$ is cyclic, $M=JM\subseteq JP\neq P$, so $M=JP$ is the Jacobson radical of $P$. If $E(T)^*\stackrel{p}{\rightarrow} S$ is the canonical projection, there is $\overline{p}:E(T)^*\rightarrow P$ with $\pi\overline{p}=p$. As $\im(\overline{p})\not\subseteq M=\ker \pi$, $\im(\overline{p})+JP=P$ so $\im(\overline{p})=P$ since $P$ is finitely generated (by Nakayama lemma). If $H=T_{E(T)}^\perp=\{f\in E(T)^*|f(x)=0,\,\forall\,x\in T\}$ is the maximal submodule of $E(T)^*$ (which is unique by \cite[Lemma 1.4]{I2}), then $H=JE(T)^*$ so $\overline{p}(H)=\overline{p}(JE(T)^*)=J\overline{p}(E(T)^*)=JP=M$.\\
Since $E(T)/X$ is finitely cogenerated by $E(T)/X\hookrightarrow C^n$, dualizing we get that the submodule $Y=X^\perp_{E(T)}=(E(T)/X)^*\hookrightarrow E(T)^*$ is finitely generated, so $\overline{p}(Y)$ is finite dimensional. If $k$ is the Loewy length of $X$, then $XJ^{k+1}=0$ and $J^{k+1}X^*=0$. Also, $X^*=E(T)^*/Y$, and $JX^*=(JE(T)^*+Y)/Y=H/Y$, so $J^k(H/Y)=0$. Now, since there is an epimorphism $H/Y\longrightarrow  \overline{p}(H)/\overline{p}(Y)$, we get that $M/\overline{p}(Y)$ has finite Loewy length. Since $\overline{p}(Y)$ is finite dimensional, it follows that $M$ has finite Loewy length, which contradicts the initial choice given by Proposition \ref{p.6}. This ends the proof.
\end{proof}

\begin{remark}\label{r.1}
It is not hard to see that the hypothesis of $E(T)/X$ being finitely cogenerated is actually equivalent to the fact that $X^\perp_{E(T)}\cong (E(T)/X)^*\subseteq E(T)^*$ is finitely generated. Indeed, if $X^\perp_{E(T)^*}$ is generated by $f_1,\dots,f_n$, then one can see that $\psi: E(T)\ni x\longmapsto (f_i(x_0)x_{-1})_{i=1,\dots,n}\in C^n$ is a morphism of left $C$-comodules (right $C^*$-modules), and $\ker(\psi)=X$, since one shows easily that $\psi(x)=0$ if and only if $f(x)=0$ for all $f\in X^\perp_{E(T)}$ (so $\ker(\psi)=(X^\perp_{E(T)})^\perp=X$). 
\end{remark}

Hence, we have the following

\begin{corollary}
Suppose the open ideals of $C^*$ are closed under products. If for all simple left $C$-comodules $T$, the maximal ideal $T^\perp$ of $C^*$ is finitely generated (equivalently, $T^\perp_{E(T)}$ is finitely generated), then left rational $C^*$-modules are closed under extensions.
\end{corollary}

We recall from \cite{CNO} that a coalgebra is called left $\Ff$-Noetherian if every closed cofinite left ideal of $C^*$ is finitely generated. The following Corollary is proved in \cite{CNO}, but also follows as a particular case of the above result:

\begin{corollary}\label{c.FN}
If $C$ is left $\Ff$-Noetherian, then left rational $C^*$-modules are closed under extensions.
\end{corollary}

The next corollary follows directly from the above results, but it gives some particularly nice and easy conditions to check in order to get that $Rat({}_{C^*}\Mm)$ is closed under extensions.

\begin{corollary}\label{c.conditions}
Let $C$ be a coalgebra such that the set of open ideals of $C^*$ is closed under products. If any of the following conditions is true, then the left rational $C^*$-modules are closed under extensions.\\
(i) For each simple right $C$-comodule $S$, its injective hull $E(S)$ has finite Loewy length (in particular, when it is finite dimensional).\\
(ii) For each simple left $C$-comodule $T$, either\\
{\bf$\bullet$} its injective hull $E(T)$ has finite Loewy length (in particular, if it is finite dimensional), or\\
{\bf$\bullet$} $L_{n+1}E(T)/L_{n}E(T)$ is finite dimensional for some $n$ (in particular, this is true when $E(T)$ is artinian), or\\
{\bf$\bullet$} $\Ext^1(L,T)\neq 0$ for only finitely many simple left comodules $L$ (which, in this case, is equivalent to: $L_1E(T)/L_0E(T)$ is finite dimensional).
\end{corollary}

We now give some applications of the above results. In particular, we note how many results in \cite{CNO} can be obtained as a corollary. We first need the following easy but useful Lemma:

\begin{lemma}\label{l.2}
Let $M$ be a left $C$-comodule such that $M$ is finitely cogenerated. Then $JM^*=M_0^\perp=(M_0)^\perp_M$. Consequently, if $M$ is a left comodule such that all $M_n=L_nM$ are finite dimensional, then $J^{n+1}M^*=M_n^\perp$ and every cofinite submodule of $M^*$ is finitely generated (and closed).
\end{lemma}
\begin{proof}
The fact that $JM\subseteq M_0^\perp$ is straightforward (it follows, for example, by \cite[Lemma 2.2]{ILMS}). Let $\sigma:M\hookrightarrow C^n$ be an embedding of left comodules. Let $f\in M_0^\perp\subseteq M^*$, and let $\alpha:M\rightarrow C$ the morphism of left $C$-comodules defined by $\alpha(m)=f(m_0)m_{-1}$. Then $\alpha\mid_{M_0}=0$ since $f\in M_0^\perp$, and so it factors as $\alpha=g\circ p$ as bellow; the following diagram is obviously commutative with $\eta$ injective since $C_0^n\cap M=M_0$ ($p$ and $\pi$ are the canonical projections). 
$$\xymatrix{
& M \ar[r]^\hookrightarrow_\sigma \ar[d]^p &  C^n\ar[d]^\pi \\
0 \ar[r] & M/M_0 \ar[d]^{g}\ar[r]^\hookrightarrow_\eta& (C/C_0)^n\ar@{..>}[dl]^{h} \\
& C & 
}$$
By the injectivity of $C$, the morphism $g$ extends to some $h$, and so we have $\alpha=gp=h\eta p=(h\pi)\sigma$. But using the standard projections $p_i$ and injections $\sigma_i$ of $C^n$, this implies that $\alpha=\sum\limits_{i=1}^n(h\pi \sigma_i)(p_i\sigma)=\sum\limits_i u_i\circ \alpha_i$ with $u_i=h\pi\sigma_i:C\rightarrow C$ and $\alpha_i=p_i\sigma:M\rightarrow C$. Therefore, if $f_i=\varepsilon\circ u_i$ and $h_i=\varepsilon\circ \alpha_i$, we have
\begin{eqnarray*}
\sum\limits_{i=1}^nf_ih_i(m) & = & \sum\limits_if_i(m_{-1})h_i(m_0)=\sum\limits_i\varepsilon(u_i(m_{-1}))\varepsilon(\alpha_i(m_0))\\
& = & \sum\limits_i\varepsilon u_i(\alpha_i(m)_1)\varepsilon(\alpha_i(m)_2) {\rm\,\,\,-\,since\,}\alpha_i{\rm\,is\,a\,comodule\,map}\\
& = & \sum\limits_i\varepsilon(u_i\alpha_i(m))=\varepsilon(\alpha(m))\\
& = & f(m)
\end{eqnarray*} 
For the last part, use induction on $n$: if $J^{n}M^*=M_{n-1}^\perp$, then $J^{n+1}M^*=JM_{n-1}^\perp$; apply the first part to $M/M_{n-1}$ and get that $J(M/M_{n-1})^*=(M_n/M_{n-1})^\perp_{M/M_{n-1}}$ which corresponds to $M_n^\perp$ through the isomorphism $(M/M_{n-1})^*\cong M_{n-1}^\perp$; therefore $JM_{n-1}^\perp=M_n^\perp$. Finally, if $I$ is a cofinite submodule of $M^*$, then $J^n\cdot (M^*/I)=0$ for some $n$, and so $J^nM^*\subseteq I$, i.e. $M_n^\perp\subseteq I$. Since $I$ contains a cofinite closed ideal, it is closed; also, since $M/M_n$ is finitely cogenerated, we have that $M_n^\perp\cong (M/M_n)^*$ is finitely generated, so $I$ is finitely generated too.
\end{proof}

We first note a generalization of \cite[Theorem 2.8]{CNO}, part of which was proved first in \cite[Theorem 4.6]{HR74}. In fact, the following also generalizes \cite[Theorem 2.11]{CNO}, and, in particular, recovers the caracterization of the commutative case. It also genegalizes some results of \cite{C0}. Recall a coalgebra $C$ is called left strongly reflexive, or $C^*$ is called almost Noetherian if every cofinite left ideal of $C^*$ is finitely generated (see \cite{HR74}). We may thus call a pseudocompact left $C^*$-module $M^*$ (i.e. $M$ is left $C$-comodule) almost Noetherian if every cofinite submodule is finitely generated.

\begin{theorem}\label{t.3}
Let $C$ be a coalgebra such that for each simple left $C$-comodule $T$, we have $\Ext^{C,1}(L,T)\neq 0$ for only finitely many simple left $C$-comodules $L$, i.e. $\Hom(L,E(T)/T)\neq 0$ for only finitely many simple left comodules $L$. Then the following assertions are equivalent:\\
(i) Rational left $C^*$-modules are closed under extensions (i.e. $C$ has a left Rat torsion functor).\\
(ii) $L_nE(T)$ is finite dimensional for all $n$ and all simple left $C$-comodules $T$.\\
(iii) $L_1E(T)$ is finite dimensional for all simple left $C$-comodules $T$.\\
(iv) $E(T)^*$ is an almost noetherian $C^*$-module for all simple left comodules $T$, i.e. every cofinite submodule of $E(T)^*$ is finitely generated.\\
(v) $C$ is locally finite.\\
(vi) $C^*$ is left $\Ff$-Noetherian.
\end{theorem}
\begin{proof}
(i)$\Rightarrow$(iii) Follows since in this case $C$ is locally finite; therefore, $\Hom(L,E(T)/T)$ is finite dimensional for all $L$; but it is also $0$ for all but finitely many left comodules $L$'s. This shows that the socle of $E(T)/T$ is finite dimensional.\\
(iii)$\Rightarrow$(ii) Follows by applying Lemma \ref{l.quasifinite} inductively.\\
(ii)$\Rightarrow$(iv) Follows from Lemma \ref{l.2}\\
(iv)$\Rightarrow$(vi) Follows since each closed cofinite left ideal $X^\perp\subset C^*$ can be decomposed as $X^\perp=X^\perp_{\bigoplus\limits_{i\in F}E(T_i)}\oplus\bigoplus\limits_{i\in H\setminus F}E(T_i)^*$, where $C=\bigoplus\limits_{i\in H}E(T_i)$ is a decomposition of $C$ into indecomposables such that $X\subseteq \bigoplus\limits_{i\in F}E(T_i)$, $F$-finite. Then one finds finite dimensional $X_i\subseteq E(T_i)$ such that $X\subseteq \bigoplus\limits_{i\in F}X_i$; we get $\bigoplus\limits_{i\in F} X_i^\perp{}_{E(T_i)}\subseteq X^\perp_{\bigoplus\limits_{i\in F}E(T_i)}\subseteq \bigoplus\limits_{i\in F}E(T_i)^*$, so $X^\perp_{\bigoplus\limits_{i\in F}E(T_i)}$ is finitely generated.\\
(vi)$\Rightarrow$(i) is known (from \cite{CNO} or above considerations: Corollary \ref{c.FN}, or first obtain (ii) and apply Corollary \ref{c.conditions}).\\
(iii)$\Leftrightarrow$(v) is a direct consequence of Lemma \ref{l.quasifinite}
\end{proof}

The hypothesis of the previous proposition says that the (left) Gabriel quiver of the coalgebra $C$ has only finitely many arrows going into any vertex $T$. In particular, the hypothesis is true if $C$ is almost connected, i.e. if it has only finitely many types of isomorphism of simple comodules. In this particular situation, we recover the above mentioned results of \cite[Theorem 2.9]{CNO}.

Another application is in the case of semiperfect coalgebras. It is proved in \cite[Lemma 2.3]{L} (see also \cite[Theorem 3.3]{NT1} and \cite[Theorem 2.12]{CNO}) that if $C$ is right semiperfect, then $C$ is $\Ff$-Noetherian and has a left Rat torsion functor. We note an alternate proof of this as a a consequence of the above results, and also a strenghtening:

\begin{corollary}
Let $C$ be a right semiperfect coalgebra. Then $C^*$ is left $\Ff$-Noetherian and $C$ has a left and right torsion functor, i.e. rational left modules and rational right modules are closed under extensions.
\end{corollary}
\begin{proof}
The fact that $C^*$ is $\Ff$-Noetherian follows directly from Theorem \ref{t.3}; therefore, finite dimensional rationals (left or right) are closed under extensions. Since $C$ is right semiperfect, left indecomposable injective comodules are finite dimensional, so Corollary \ref{c.conditions} (ii) implies that left rationals are closed under extensions and Corollary \ref{c.conditions} (i) (its left-right symmetric version), implies that right rationals are closed under extensions too.
\end{proof}

We will see in Example \ref{e.2} that for a right semiperfect coalgebra $C$, its dual $C^*$ need not be right $\Ff$-Noetherian. Another direct consequence of Theorem \ref{t.1} (but which can be easily obtained also directly, see also \cite{L2}) is

\begin{corollary}
Let $C=\bigoplus\limits_{i\in I}C_i$ be a direct sum of coalgebras. $C$ has a left rational torsion functor if and only if all $C_i$ have left rational torsion functors.
\end{corollary}


\section{$\Ff$-Noetherian coalgebras and torsion rational functor}

We give connections between $\Ff$-Noetherian and properies investigated before, more specifically, locally finiteness. For a simple left comodule $S$ and a comodule $M$, let $[M;S]$ denote the multiplicity of $S$ in the socle of $M$. It can be finite or infinite. Note that a comodule $M$ is quasifinite if $[M;S]<\infty$ for all simples $S$. The following proposition gives a cryterion to test the $\Ff$-Noetherian property:

\begin{proposition}\label{p.4}
The following assertions are equivalent:\\
(i) $C^*$ is left $\Ff$-Noetherian.\\
(ii) $C/X$ is finitely cogenerated for each finite dimensional left subcomodule $X\subseteq C$, that is, there is a monomorphism $C/X\hookrightarrow C^n$ for some $n$.\\
(iii) $\sup\{\frac{[C/X;S]}{[C;S]}|S {\rm\, simple\,left\,comodule}\}<\infty$ for each finite dimensional left subcomodule $X$ of $C$.
\end{proposition}
\begin{proof}
(i)$\Leftrightarrow$(ii) We proceed similar to Remark \ref{r.1}. If $C/X\hookrightarrow C^n$ is a monomorphism, dualizing we get an epimorphism $(C^*)^n\rightarrow M^*\rightarrow 0$. Conversely, let $f_1,\dots,f_n$ generate the right $C^*$-module $X^\perp$, and let $X_i=(C^*f_i)^\perp$. Then $X^\perp=\sum\limits_{i=1}^nX_i^\perp=(\bigcap\limits_{i=1}^nX_i)^\perp$ so $X=\bigcap\limits_{i=1}^nX_i$. The map $\varphi_i:C\rightarrow C$, $\varphi(c)=f_i\cdot c$ is a morphism of right $C^*$-modules (so of left $C$-comodules), and $\ker(\varphi_i)=\{c|c_1f_i(c_2)=0\}=\{c|c^*(c_1f_i(c_2))=0\,\forall c^*\in C^*\}=(C^*f_i)^\perp=X_i$. This shows that $C/X_i$ embedds in $C$. Therefore we can get an embedding
$$C/X=C/\bigcap\limits_{i=1}^nX_i\hookrightarrow \bigoplus\limits_{i=1}^nC/X_i\hookrightarrow C^n$$
(i)$\Leftrightarrow$(iii) If $C/X$ embedds in $C^n$ then the multiplicity of $S$ in $C/X$ is smaller than in $C^n$, so $[C/X;S]\leq n[C;S]$. Conversely, if $[C/X;S]\leq n[C;S]$ for some fixed $n$ (depending on $X$) and all $S$, we see that the socle of $C/X$ embedds in $C_0^n$: $s(C/X)=\bigoplus\limits_{S}S^{[C/X;S]}\hookrightarrow\bigoplus\limits_{S}S^{n[C;S]}=C_0^n$. Hence $s(C/X)$ embedds in $C^n$, so $C/X$ embedds in $C^n$ since $s(C/X)$ is essential in $C/X$ and $C^n$ is injective. This finishes the proof. 
\end{proof}

It was conjectured in \cite[Remark 2.13]{CNO} that if $C$ has a left torsion Rat-functor, i.e. the left rational are closed under extensions, then $C$ is $\Ff$-Noetherian. The conjecture was motivated by a series of results on the Rat functor which gave evidence for it. However, it was proved in \cite{TT} by using a certain more complicated semigroup coalgebra construction that this conjecture is false. This simple characterization of $\Ff$-Noetherian, together with a result of Radford from \cite{Rad3}, yield an easy way to give counterexamples to this conjecture

\begin{example}\label{e.2}
Consider the following quiver $\Gamma$:
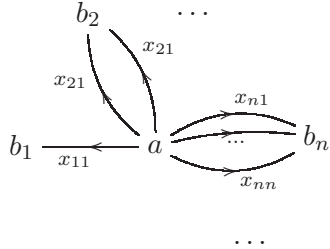
$$\xygraph{ 
!~:{@{-}|@{>}}
a((:[l(1.4)]{b_1}^(.6){x_{11}}[uur(.7)]{b_2}[r(1.1)]{\dots}[dr(1.3)]{b_n}),:@/^.7pc/"b_2"^(.6){x_{21}},:@/_.7pc/"b_2"_(.6){x_{21}},:@/^.9pc/"b_n"^(.6){x_{n1}},
!~:{@{}|@{}} :"b_n"|{\dots},!~:{@{-}|@{>}} :@/_.9pc/"b_n"_(.6){x_{nn}}, :@/^.3pc/"b_n",[dr]{\dots})
}$$
Let $C$ be the path subcoalgebra of the full path coalgebra of $\Gamma$, which has as a $\KK$-basis the set $\{a,b_n,x_{ni}|n\in\NN,1\leq i\leq n\}$. We see that this coalgebra has $C=C_1$. Arguing exactly as we did in Example \ref{e.1}, we see that this coalgebra is locally finite (or apply Lemma \ref{l.quasifinite} and note that the space of $(u,v)$ skew primitives is finite dimensional for any two grouplike elements $u,v$). As pointed out before, if $\KK$ is infinite then $C_0\cong \KK^{(\NN)}$ is coreflexive. Then, by \cite[3.3]{Rad3}, since $C_0$ is coreflexive, $C$ is locally finite and $C=C_1$, we get that $C$ is coreflexive (note: this is called reflexive in \cite{Rad3}). By the results of the previous sections (for example, Proposition \ref{p.reflexive} or Corollary \ref{c.conditions}), $C$ will have a torsion Rat-functor (for both left and right $C^*$-modules). But note that $[C/\KK\{a\};\KK\{b_n\}]=n$ as right comodules, and $C$ is pointed so $[C;\KK\{b_n\}]=1$, so condition (iii) of Proposition \ref{p.4} is not verified for the right comodule $\KK\{a\}$. Therefore, by (the right version) of Proposition \ref{p.4}, we have that $C^*$ is not right $\Ff$-Noetherian. We note that at each vertex $u$, there are only finitely many arrows (and paths) into $u$, so by well known characterizations of the injective indecomposable objects over path coalgebras (see \cite{simson}, \cite{DIN}) it follows that this left injective indecomposable are finite dimensional. Therefore, $C$ is right semiperfect. However, infinitely many arrows go out of $a$, so the injective hull of the right comodule $\KK\{a\}$ is infinite dimensional. Hence, $C$ is not left semiperfect. This example shows also that if $C$ is a left (right) semiperfect coalgebra then $C^*$ is not necessarily left (right) $\Ff$-Noetherian, and also, that $C^*$ can be only one-sided $\Ff$-Noetherian, so it makes sense to distinguish between the two notions.
\end{example}

{\bf Closing Remarks} We have seen that in general, the rational modules are closed under extensions if and only if the finite dimensional rationals are so, and a homological condition holds. We note that in all the particular classes of examples that we have for which the rationals are closed under extension (e.g. see Corollary \ref{c.conditions}, Theorem \ref{t.3}, and their consequences), the Ext-condition from Theorem \ref{t.1} is automatically satisfied. That is, if finite rationals are closed under extensions, then all left $C^*$-modules are so. The known examples for which the $Rat$ functor is not torsion, are in fact examples where the finite rational modules are not closed under extensions. One can see that the counter-example to the above mentioned question \cite[2.13]{CNO} build in \cite[Example 12]{TT} is also of this type. Indeed, one can see that the coalgebra $C=\KK S$ in that example has coradical filtration of length 2 (i.e. $C=C_2$), and so by Corollary \ref{c.conditions}, the finite rationals are not closed under extensions; in fact, this coalgebra $C$ is not coreflexive, so the example is of the same nature as the above Example \ref{e.2}. This leads one to ask the following questions:

\vspace{.5cm}

{\bf Question 1} If the finite dimensional rational (left, or equivalently, right) $C^*$-modules are closed under extensions, does it follow that $Rat({}_{C^*}\Mm)$ (and, by symmetry, also $Rat(\Mm_{C^*})$) is closed under extensions?

\vspace{.5cm}

{\bf Question 2} If $Rat({}_{C^*}\Mm)$ is closed under extensions, does it also follow that $Rat(\Mm_{C^*})$ is closed under extensions? 

\vspace{.5cm}

{\bf Question 3} If $E$ is an injective indecomposable left comodule, and $S$ is a simple left $C$-comodule, does it follow that $\Ext^1(S,E)=0$? 

\vspace{.5cm}

Note that an afirmative answer to Q1 implies afirmative an answer to Question 2, and afirmative answer to Q3 implies afirmative Q1. We believe that a counterexample to Q1, Q2 and Q3 might be constructed by using the characterizations of Theorem \ref{t.3}, more precisely, a coalgebra for which for all simple left comodules $S$, $\Ext^{C,1}(L,S)\neq 0$ for only finitely many simple left comodules $L$, but for which this condition is not true for right simple comodules. Such a coalgebra might be obtained by considering path (sub)coalgebras of quiver coalgebras for quivers $\Gamma$ in which at each vertex there are only finitely many arrows going in, but for some vertices there are infinitely arrows going out. In this case, $C$ would have the following properties: finite dimensional rational modules are closed under extensions, and $C^*$ is left $\Ff$-Noetherian, so $Rat({}_{C^*}\Mm)$ is closed under extensions; but one might expect that $Rat(\Mm_{C^*})$ is not closed under extensions, which would answer Q1,Q2 and Q3 in the negative.

\bigskip\bigskip\bigskip

\begin{center}
\sc Acknowledgment
\end{center}
The author wishes to thank E.Friedlander for interesting insights and conversations on the theory of rational modules.
This work was supported by the strategic grant POSDRU/89/1.5/S/58852, Project ``Postdoctoral programe for
training scientific researchers'' cofinanced by the European Social Fund within the Sectorial Operational Program Human
Resources Development 2007-2013.






\end{document}